\documentclass[reqno]{amsart}
\usepackage[utf8]{inputenc}
\usepackage{amssymb}
\usepackage{amsmath}
\usepackage{hyperref}
\usepackage{enumitem}
\usepackage{mathtools}
\newtheorem{theorem}{Theorem}
\newtheorem{lemma}[theorem]{Lemma}

\newtheorem{corollary}[theorem]{Corollary}
\theoremstyle{definition}

\numberwithin{theorem}{section}
\numberwithin{equation}{section}

\newcommand{\op}[1]{\operatorname{#1}}
\newcommand{\Pois}{\op{Pois}}
\newcommand{\ESF}{\op{ESF}}

\renewcommand\l{\left}
\renewcommand\r{\right}
\newcommand{\floor}[1]{{\l\lfloor#1\r\rfloor}}
\newcommand{\ceil}[1]{{\l\lceil#1\r\rceil}}

\begin{document}

%%% TITLE %%%
\title{Random generation under the Ewens distribution}
\author{Sean Eberhard}
\address{Sean Eberhard, London, UK}
\email{eberhard.math@gmail.com}

%%% ABSTRACT %%%
\begin{abstract}
The Ewens sampling formula with parameter $\alpha$ is the distribution on $S_n$ which gives each $\pi\in S_n$ weight proportional to $\alpha^{C(\pi)}$, where $C(\pi)$ is the number of cycles of $\pi$. We show that, for any fixed $\alpha$, two Ewens-random permutations generate at least $A_n$ with high probability. More generally we work out how many permutations are needed for $\alpha$ growing with $n$. Roughly speaking, two are needed for $0 \leq \alpha \ll n^{1/2}$, three for $n^{1/2} \ll \alpha \ll n^{2/3}$, etc.
\end{abstract}

\maketitle

\section{Introduction}

The Ewens sampling formula with parameter $\alpha\geq 0$ is the distribution on $S_n$ which gives each $\pi \in S_n$ weight proportional to $\alpha^{C(\pi)}$, where $C(\pi)$ is the number of cycles of $\pi$. To be explicit, we say that $\pi$ has distribution $\ESF(\alpha, n)$ if
\[
  P(\pi = \sigma) = \frac{\alpha^{C(\sigma)}}{\alpha^{(n)}} \qquad(\sigma \in S_n),
\]
where 
\[
  \alpha^{(n)} = \alpha (\alpha+1) \cdots (\alpha + n-1).
\]
Thus when $\alpha = 1$ we have simply the uniform distribution, if $\alpha=0$ we have the uniform distribution on $n$-cycles, and as $\alpha\to \infty$ the distribution tends towards a point mass at the identity. In general, a larger $\alpha$ gives $\pi$ more of a tendency to have many cycles.

The Ewens sampling formula can be motivated from many different perspectives. See Crane~\cite{crane} for a survey. Ewens used it to model the frequency of alleles of a given neutral gene in a population undergoing natural selection (see \cite{ewens}). Mathematically, this is a consequence of another perspective: we can think of $\ESF(\alpha,n)$ as specifying the cycle type of a permutation according to independent Poisson random variables $Z_i \sim \Pois(\alpha / i)$ ($1\leq i \leq n$) conditional on $\sum_{i=1}^n i Z_i = n$. See Arratia, Barbour, and Tavar\'e~\cite{arratia-barbour-tavare} for much more from this perspective. Another motivation comes from the analogy with number theory: $\ESF(\alpha, n)$ is analogous to weighting integers $x$ by $\alpha^{\omega(x)}$, where $\omega(x)$ is the number of prime factors of $x$ counting multiplicity (see for example Hall and Tenenbaum~\cite{hall-tenenbaum}, particularly Sections~0.5 and 4.3).

Our own motivation is simply that $\ESF(\alpha, n)$ generalizes the uniform distribution in a simple and tractable way, and lends a fresh perspective to questions about random generation. Dixon~\cite{dixon1} proved that if $\pi_1, \pi_2 \in S_n$ are chosen uniformly at random then we have $\langle \pi_1, \pi_2 \rangle \geq A_n$ asymptotically almost surely. The purpose of the present note is to generalize this theorem to the Ewens distribution, and to see how the result depends on $\alpha$. The main assertion is the following.

\begin{theorem} \label{main-theorem}
Fix $t\geq 2$, and let $\alpha=\alpha(n)\geq 0$. Draw $\pi_1, \dots, \pi_t \sim \ESF(\alpha,n)$ independently. Then the probability that $\langle\pi_1, \dots, \pi_t\rangle \geq A_n$ is
\[
  e^{-n(\alpha/n)^t} + O\l( 1/ \log n \r).
\]
The constant implicit in the error term depends on $t$ but not on $\alpha$.
\end{theorem}

\begin{corollary} \label{main-corollary}
Fix $t\geq 2$. Suppose $\alpha = p n^\theta$ for constant $p, \theta \geq 0$. Draw $\pi_1, \dots, \pi_t \sim \ESF(\alpha, n)$ independently. Then
\[
  P(\langle \pi_1, \dots, \pi_t \rangle \geq A_n) \longrightarrow
  \begin{cases}
    1 & \textup{if}~\theta < 1-1/t, \\
    0 & \textup{if}~\theta > 1-1/t, \\
    e^{-p^t} & \textup{if}~\theta=1-1/t.
  \end{cases}
\]
\end{corollary}

In words, two permutations continue to generate with high probability for any constant $\alpha$, and even for $\alpha$ up to roughly $n^{1/2}$. For $n^{1/2} \ll \alpha \ll n^{2/3}$, three permutations are needed, for $n^{2/3} \ll \alpha \ll n^{3/4}$ four are needed, and so on. For very large $\alpha$, specifically for $\alpha \geq n^{1-o(1)}$, any bounded number of $\ESF(\alpha,n)$ permutations will fail to generate.

Recently, Brito, Fowler, Junge, and Levy~\cite{brito_fowler_junge_levy_2018} studied the Ewens distribution in the context of invariable generation, and it is interesting to compare Theorem~\ref{main-theorem} with their result. Recall that permutations $\pi_1, \dots, \pi_t \in S_n$ are said to \emph{invariably generate} if for all $g_1, \dots, g_t \in S_n$ we have
\[
  \langle \pi_1^{g_1}, \dots, \pi_t^{g_t} \rangle = S_n.
\]
It was recently proved in \cite{pemantle-peres-rivin, eberhard-ford-green-invariable} that the minimal number of permutations which invariably generate with probability bounded away from zero is exactly four. Generalizing this (as well a significant amount of relevant background, with impressive efficiency) to the Ewens distribution, Brito, Fowler, Junge, and Levy proved the following theorem (see \cite[Theorem~1]{brito_fowler_junge_levy_2018}).

\begin{theorem}[Brito--Fowler--Junge--Levy, 2018] \label{bfjl-theorem}
Let $\alpha\geq 0$ be fixed. Then the minimal number of independent $\ESF(\alpha, n)$ needed so that the probability of invariable generation is bounded away from zero is exactly
\begin{equation} \label{bfjl-formula}
  h(\alpha) = \ceil{(1 - \alpha \log 2)^{-1}},
\end{equation}
provided that $\alpha < 1/\log 2$ and that $h(\alpha)$ is not an integer. If $h(\alpha)$ is an integer then that minimal number is either $h(\alpha)$ or $h(\alpha)+1$, while if $\alpha \geq 1/\log 2$ then no bounded number of $\ESF(\alpha,n)$ permutations are enough.
\end{theorem}

While for ordinary generation Theorem~\ref{main-theorem} asserts that a bounded number of permutations are enough until around $\alpha = n^{1-o(1)}$, for invariable generation the required number of permutations blows up already at $\alpha = 1/\log 2$. This is a substantial contrast, though not really surprising, given that invariable generation is such a stronger property: the main obstruction to ordinary generation is the existence of a common fixed point, while the main obstruction to invariable generation is the existence of fixed subsets of a common size.

A few questions are left unanswered by Theorems~\ref{main-theorem} and \ref{bfjl-theorem}. For the latter theorem, which of $h(\alpha)$ and $h(\alpha)+1$ is correct at points of discontinuity (\cite[Question~1]{brito_fowler_junge_levy_2018})? And how many permutations (as a function of $n$) are needed for a given $\alpha \geq 1/ \log 2$ (\cite[Question~3]{brito_fowler_junge_levy_2018})? (The authors conjecture $\beta \log n$ for some $\beta = \beta(\alpha)$.) Analogously, for Theorem~\ref{main-theorem}, while the behaviour at points of discontinuity is understood, it would be interesting to know how many permutations are needed as a function of $n$ for $\alpha \geq n^{1-o(1)}$.

\subsection{Notation}

All of our nonstandard notation appears in the line
\[
  P_\alpha(\pi = \sigma) = \frac{\alpha^{C(\sigma)}}{\alpha^{(n)}} \qquad (\pi \sim \ESF(\alpha, n), \sigma \in S_n),
\]
which incidentally defines $\ESF(\alpha, n)$. We write $C(\sigma)$ for the number of cycles in $\sigma$, and we use the notation $\alpha^{(n)}$ for the ``rising factorial''
\[
  \alpha^{(n)} = \alpha (\alpha+1) \cdots (\alpha+n-1).
\]

Random permutations will be denoted $\pi$ or $\pi_1, \dots, \pi_t$, and act on the set $\Omega = \{1, \dots, n\}$, which tends to have subsets called $X$. We subscript $P$s and $E$s by $\alpha$s to indicate that the $\pi$s are taken from $\ESF(\alpha,n)$.

Throughout the paper we assume $t \geq 2$ is a fixed integer. Constants implicit in big-O notation may depend on $t$, but never on $\alpha$.

\section{Transitive subgroups}

Permutations $\pi_1, \dots, \pi_t$ fail to generate $S_n$ or $A_n$ if and only if there is a subgroup $H \notin \{S_n, A_n\}$ such that $\pi_1, \dots, \pi_t \in H$. Theorem~\ref{main-theorem} will be proved by ruling out every possible subgroup $H$. We begin with the transitive subgroups.

\begin{lemma}
Assume $\alpha \geq 1$. Let $k \geq 10 \alpha (\log n + 1)$. Then
\[
  P_\alpha(C(\pi) = k) \leq e^{-k}.
\]
\end{lemma}
\begin{proof}
Let $c(n,k)$ be the number of $\pi\in S_n$ with exactly $k$ cycles. It easy to prove that
\begin{equation}\label{cnk-bound}
  \frac{c(n,k)}{n!} \leq \frac{(\log n + 1)^k}{k!}.
\end{equation}
For instance, by counting pairs $(c, \pi)$ such that $c$ is a cycle of $\pi$ in two different ways, we have the recurrence
\[
  k c(n,k) = \sum_{j=1}^n \frac{n!}{j (n-j)!} c(n-j, k-1),
\]
whence \eqref{cnk-bound} follows easily by induction.

Written differently, we have
\[
  P_\alpha (C(\pi)=k) \leq \frac{n!}{\alpha^{(n)}} \frac{\alpha^k (\log n+1)^k}{k!}.
\]
Since $\alpha \geq 1$ we have $n!/\alpha^{(n)} \leq 1$. Thus for $k \geq 10 \alpha \log n$ we have
\[
  P_\alpha(C(\pi)=k) \leq \frac{\alpha^k (\log n + 1)^k}{k!} \leq \left( \frac{\alpha (\log n + 1)}{k/e} \right)^k \leq e^{-k}. \qedhere
\]
\end{proof}

\begin{theorem}\label{transitive}
Assume $0\leq \alpha \leq \epsilon n / \log^2 n$, where $\epsilon = 10^{-4}$. Let $\pi_1, \pi_2 \sim \ESF(\alpha, n)$. Then the probability that $\langle \pi_1, \pi_2\rangle$ is transitive and different from $S_n$ or $A_n$ is $O(e^{-cn / \log n})$.
\end{theorem}
\begin{proof}
Let $E$ be the event that $\langle \pi_1, \pi_2\rangle$ is transitive and different from $S_n$ or $A_n$. By Dixon~\cite[Lemma~2]{dixon1} and Babai~\cite[Theorem~1.4]{babai} we have
\begin{equation}\label{P1T}
  P_1(E) \leq n 2^{-n/4} + n^{\sqrt{n}}/n! = O( n 2^{-n/4} ).
\end{equation}
This bound is strong enough that we can deduce a bound for $P_\alpha(E)$ for general $\alpha$.

First suppose $\alpha \leq 1$. The density of $\ESF(\alpha,n)$ with respect to $\ESF(1, n)$ is
\[
  \frac{\alpha^{C(\pi)} n!}{\alpha^{(n)}} \leq \frac{\alpha n!}{\alpha^{(n)}}
  = \frac{n!}{(\alpha+1) \cdots (\alpha+n-1)}
  \leq n.
\]
Thus by \eqref{P1T} we have
\[
  P_\alpha(T) \leq n^2 P_1(T) = O(e^{-cn}). 
\]

Now suppose $\alpha \geq 1$. Let $B$ be the event that $\pi_1$ or $\pi_2$ has more than $n/(100 \log n)$ cycles. Let $E_1 = E \cap B$ and $E_2 = E \setminus B$. By the lemma we have
\[
  P_\alpha(E_1) \leq P_\alpha(B) = O\left( e^{-n/(100 \log n)} \right).
\]
On the other hand the density of $\ESF(\alpha, n)$ with respect to $\ESF(1,n)$ on $B^c$ is
\[
  \frac{\alpha^{C(\pi)} n!}{\alpha^{(n)}} \leq \alpha^{n/(100 \log n)} \leq n^{n/(100 \log n)} = e^{n/100}.
\]
Thus by \eqref{P1T} we have
\[
  P_\alpha(E_2) \leq e^{n/50} P_1(T) \leq e^{-cn}.
\]
This completes the proof.
\end{proof}

\section{Intransitive subgroups}

Now we turn to intransitive subgroups $H$. Every maximal intransitive subgroup of $S_n$ is isomorphic to $S_k \times S_{n-k}$ for some $k$, given as the setwise stabilizer of some subset $X$ of size $k$. Conveniently, there is a neat explicit formula for
\[
  P_\alpha(\pi \in S_k \times S_{n-k}).
\]

\begin{lemma}\label{S_k-lemma}
Let $X \subset \Omega$ be a set of size $k$. Then
\[
  P_\alpha(\pi(X) = X) =
  \frac{\alpha^{(k)} \alpha^{(n-k)}}{\alpha^{(n)}}.
\]
\end{lemma}
\begin{proof}
\[
  \sum_{\pi(X) = X} \alpha^{C(\pi)} = \sum_{(\pi_1, \pi_2) \in S_k \times S_{n-k}} \alpha^{C(\pi_1) + C(\pi_2)} = \alpha^{(k)} \alpha^{(n-k)}.\qedhere
\]
\end{proof}

\begin{lemma}\label{N_k-lemma}
Let $N_k$ be the number of $k$-sets fixed simultaneously by $\pi_1, \dots, \pi_t$. We have the following estimates:
\begin{enumerate}[label=\upshape{(\alph*)}]
  \item $E_\alpha(N_k) = \binom{n}{k} \left( \frac{\alpha^{(k)} \alpha^{(n-k)}}{\alpha^{(n)}} \right)^t$; in particular $E_\alpha(N_1) = n \left( \frac{\alpha}{n+\alpha-1}\right)^t$; \label{E_alpha(N_1)}
  \item $E_\alpha(N_k)$ is monotonically decreasing in $k$ for $\alpha-1 \leq k < n/2$; \label{N_k-monotonic}
  \item $E_\alpha(N_k) \leq \frac1{k!} E_\alpha(N_1)^k \cdot e^{t k^2/\alpha}$; \label{E_alpha(N_k)-bound}
  \item $\sum_{k = 2}^\floor{n/2} E_\alpha(N_k) \leq g\l(e^t E_\alpha(N_1) \r) + O\l( n^{-2} \r)$, where $g(x) = e^x - 1 - x$, provided that $\alpha \leq n/100$. \label{sum-N_k-bound}
\end{enumerate}
\end{lemma}
\begin{proof}
Part \ref{E_alpha(N_1)} is clear from Lemma~\ref{S_k-lemma}. To prove \ref{N_k-monotonic}, note
\begin{align*}
  \frac{E_\alpha(N_{k+1})}{E_\alpha(N_k)}
  = \frac{n-k}{k+1}
    \l( \frac{\alpha + k}{\alpha + n-k - 1} \r)^t = \frac{f(k+1)}{f(n-k)},
\end{align*}
where
\[
  f(x) = \frac{(x + \alpha - 1)^t}x .
\]
It suffices to show that $f(x)$ is increasing for $x \geq \alpha-1$. We have
\[
  (\log f)'(x) = \frac{t}{x + \alpha - 1} - \frac1x,
\]
so in fact $f$ is increasing for $x \geq (\alpha-1)/(t-1)$.

For \ref{E_alpha(N_k)-bound}, we have
\begin{align*}
  E_\alpha(N_k)
  &= \binom{n}{k} \l( \frac{\alpha^{(k)} \alpha^{(n-k)} } {\alpha^{(n)}} \r)^t \\
  &= \binom{n}{k} \l( \frac{\alpha \cdots (\alpha+k-1)}{(\alpha + n-k) \cdots (\alpha + n-1)} \r)^t \\
  &\leq \frac{n^k}{k!} \l( \frac{\alpha + k-1}{\alpha + n-1} \r)^{kt}.
\end{align*}
Thus
\[
  \frac{E_\alpha(N_k)}{E_\alpha(N_1)^k/k!} \leq \l( 1 + \frac{k-1}{\alpha} \r)^{kt} \leq e^{tk^2 / \alpha}.
\]

Now we turn to \ref{sum-N_k-bound}. First suppose $\alpha \leq 10$. Then by \ref{E_alpha(N_1)} we have $E_\alpha(N_k) = O(n^{-(t-1)k})$, so by \ref{N_k-monotonic} we have
\begin{align*}
  \sum_{k=2}^\floor{n/2} E_\alpha(N_k) \leq \sum_{k=2}^8 E_\alpha(N_2) + n E_\alpha(N_9) = O(n^{-2(t-1)} + n^{1-9(t-1)}).
\end{align*}
Since $t \geq 2$, this is $O(n^{-2})$.

Now assume $\alpha \geq 10$. Let $k_0=\floor{\alpha}$. Note from \ref{E_alpha(N_k)-bound} that for $k \leq \alpha$ we have
\[
  E_\alpha(N_k) \leq \frac1{k!} E_\alpha(N_1)^k e^{tk^2/\alpha} \leq \frac1{k!} E_\alpha(N_1)^k e^{tk}.
\]
Thus
\[
  \sum_{k=2}^{k_0} E_\alpha(N_k) \leq g\l( e^t E_\alpha(N_1) \r).
\]
On the other hand by \ref{N_k-monotonic} we have
\begin{equation}\label{sum-from-k_0}
  \sum_{k=k_0}^\floor{n/2} E_\alpha(N_k) \leq n E_\alpha(N_{k_0}) \leq  \frac{n}{k_0!} E_\alpha(N_1)^{k_0} e^{tk_0}
  \leq n \l( e^{t+1} E_\alpha(N_1) / k_0 \r)^{k_0}.
\end{equation}
Now from \ref{E_alpha(N_1)} we have
\[
  e^{t+1} E_\alpha(N_1) / k_0 \leq e^{t+2} (\alpha / n)^{t-1},
\]
and thus \eqref{sum-from-k_0} is bounded by
\[
  n \l( e^{t+2} (\alpha / n)^{t-1} \r)^{\alpha - 1}.
\]
It is easy to see that this expression is maximized in the range $10 \leq \alpha \leq n/100$ at $\alpha=10$, where it is $O(n^{-8})$, so we're done.
\end{proof}

\begin{lemma} \label{N_k*-lemma}
Let $N_k^*$ be the number of $k$-sets fixed simultaneously by $\pi_1, \dots, \pi_t$ on which $\langle\pi_1, \dots, \pi_t\rangle$ acts transitively. Then
\begin{enumerate}[label=\upshape{(\alph*)}]
  \item $E_\alpha(N_k^*) \leq E_\alpha(N_k) \cdot tk/\alpha$ for $k > 1$; \label{E_alpha(N_k^*)-bound}
  \item $\sum_{k=2}^\floor{n/2} E_\alpha(N_k^*) = O(n^{-1/4})$, provided that $E_\alpha(N_1) \leq e^{-t-10} \log n$. \label{sum-E_alpha(N_k*)-bound}
\end{enumerate}
\end{lemma}
\begin{proof}
For \ref{E_alpha(N_k^*)-bound}, note that $E_\alpha(N_k^*) \leq E_\alpha(N_k) \cdot P_\alpha(T_k)$, where $T_k$ is the event that $\sigma_1, \dots, \sigma_t \sim \ESF(\alpha, k)$ generate a transitive group. Of course, this is the very thing we are trying to estimate in this section, but for our present purpose it suffices to crudely bound $P_\alpha(T_k)$ by the probability that at least one of $\sigma_1, \dots, \sigma_t$ moves some marked point $1 \in \{1,\dots,k\}$ (using $k > 1$). From Lemma~\ref{S_k-lemma} we have
\[
  P_\alpha(\sigma_i \cdot 1 \neq 1) = \frac{k-1}{\alpha + k - 1},
\]
so indeed
\[
  P_\alpha(T_k) \leq t \cdot \frac{k-1}{\alpha+k-1} \leq tk/\alpha.
\]
This proves \ref{E_alpha(N_k^*)-bound}.

The proof of \ref{sum-E_alpha(N_k*)-bound} is much like the proof of Lemma~\ref{N_k-lemma}\ref{sum-N_k-bound}. Again we may assume $\alpha \geq 10$. Arguing as before, we have
\[
  \sum_{k=2}^\floor{n/2} E_\alpha(N_k^*) \leq
  \sum_{k=2}^{k_0} \frac1{k!} E_\alpha(N_1)^k e^{tk} \cdot tk/\alpha + n \l( e^{t+1} E_\alpha(N_1) / k_0 \r)^{k_0},
\]
where $k_0 = \floor{\alpha}$. The latter term may be bounded exactly as before (noting that we must have $\alpha \leq n/100$ given the hypothesis about $E_\alpha(N_1)$), while the sum is bounded by
\[
  h\l( e^t E_\alpha(N_1) \r) \cdot t/\alpha,
\]
where
\[
  h(x) = \sum_{k=2}^\infty \frac{x^k}{(k-1)!} = x (e^x - 1).
\]
If $\alpha \leq n^{1/3}$ then we have
\[
  E_\alpha(N_1) \leq n(\alpha/n)^2 \leq n^{-1/3},
\]
and thus
\[
  h\l(e^t E_\alpha(N_1)\r) \cdot t/\alpha \leq O\l( n^{-2/3} \r).
\]
On the other hand if $\alpha \geq n^{1/3}$ and $e^t E_\alpha(N_1) \leq \epsilon \log n$ then we have
\[
  h\l(e^t E_\alpha(N_1)\r) \cdot t/\alpha \leq n^{-1/3 + \epsilon + o(1)}. 
\]
Thus either way we have the bound we need.
\end{proof}

The point of the previous two lemmas is that the probability that $\langle \pi_1, \dots, \pi_t \rangle$ is transitive is controlled by the probability that $N_1=0$. It therefore remains only to understand the behaviour of $N_1$.

Analogously to Lemma~\ref{S_k-lemma} and Lemma~\ref{N_k-lemma}\ref{E_alpha(N_1)}, we have
\[
  E_\alpha\l( \binom{N_1}{k} \r) = \binom{n}{k} \l( \frac{\alpha^k \alpha^{(n-k)}}{\alpha^{(n)}} \r)^t.
\]
If $n (\alpha/n)^t \to x$ then this converges to $x^k/k!$. It follows by the method of moments that $N_1$ converges in distribution to $\Pois(x)$. We need a version of this argument with an explicit error term. The following special case suffices.

\begin{lemma} \label{bonferroni}
The following estimate holds:
\[
  P_\alpha(N_1 = 0) = e^{-E_\alpha(N_1)} + O(1/\log n).
\]
\end{lemma}
\begin{proof}
Note that
\begin{align}
\frac{\alpha^k \alpha^{(n-k)}}{\alpha^{(n)}}
  &= \l( \frac{\alpha}{n + \alpha - 1} \r)^k
\l(
  \l(1 - \frac1{n + \alpha - 1} \r) \dots \l( 1 - \frac{k-1}{n + \alpha - 1} \r)
\r)^{-1} \nonumber \\
  &= \l( \frac{\alpha}{n+\alpha -1} \r)^k e^ { O\l( k^2/n \r) }. \label{alpha-technical-calc}
\end{align}
Thus
\begin{align}
  E_\alpha\l( \binom{N_1}{k} \r)
  &= \binom{n}{k} \l( \frac{\alpha^k \alpha^{(n-k)}}{\alpha^{(n)}} \r)^t \nonumber \\
  &= \binom{n}{k} \l( \frac{\alpha}{n+\alpha - 1} \r)^{kt} e^{O(tk^2/n)} \nonumber \\
  &= \frac{E_\alpha(N_1)^k}{k!} e^{O(tk^2/n)} \label{E_alpha-binom-N_1}
\end{align}
(the last equalitiy using also the $\alpha=1$ case of \eqref{alpha-technical-calc}). 
Thus it follows from Bonferroni's inequalities (a.k.a., inclusion--exclusion) that, for any $r$,
\begin{align*}
  P_\alpha(N_1 = 0)
  &= \sum_{k=0}^{r-1} (-1)^k E_\alpha\l( \binom{N_1}{k} \r) + O\l( E_\alpha\l( \binom{N_1}{r} \r) \r) \\
  &= \sum_{k=0}^{r-1} (-1)^k \frac{E_\alpha(N_1)^k}{k!} e^{t k^2 /n} + O\l( \frac{E_\alpha(N_1)^r }{r!} e^{O(t r^2 /n)} \r) \\
  &= e^{-E_\alpha(N_1)} + O\l( e^{E_\alpha(N_1)} (e^{t r^2/n} - 1) + \frac{E_\alpha(N_1)^r}{r!} e^{O(t r^2/n)}  \r).
\end{align*}

Assume first $E_\alpha(N_1) \leq \frac{1}{2} \log n$. Then we may take $r \sim 100 \log n$ and have
\[
\frac{E_\alpha(N_1)^r}{r!} \leq (e E_\alpha(N_1)/r)^r \leq n^{-c},
\]
while also
\[
  e^{E_\alpha(N_1)} (e^{t r^2/n} - 1) \leq n^{1/2} \cdot O(t r^2/n) = O\l( n^{-c} \r).
\]
Thus we have
\[
  P_\alpha(N_1=0) = e^{-E_\alpha(N_1)} + O(n^{-c})
\]
in this case.

Now assume $E_\alpha(N_1) \geq \frac{1}{2} \log n$. In this case the above method is ineffective, but we can use the second moment method instead. From \eqref{E_alpha-binom-N_1} we have
\[
  E_\alpha(N_1^2) = E_\alpha(N_1) + 2 E_\alpha\l( \binom{N_1}{2} \r) = E_\alpha(N_1) + E_\alpha(N_1)^2 e^{O(t/n)}.
\]
Thus by Chebyshev's inequality we have
\[
  P_\alpha(N_1 = 0) \leq \frac{\op{Var}_\alpha(N_1)}{E_\alpha(N_1)^2} = \frac1{E_\alpha(N_1)} + e^{O(t/n)} - 1 = \frac1{E_\alpha(N_1)} + O(n^{-1}).
\]
Thus
\[
  P_\alpha(N_1 = 0) = O(1/\log n),
\]
as required.
\end{proof}

Finally we are ready to estimate the probability that $\langle \pi_1, \dots, \pi_t\rangle$ is transitive.

\begin{theorem}\label{intransitive}
The probability that $\langle \pi_1, \dots, \pi_t\rangle$ is transitive is
\[
  e^{ -n (\alpha / n)^t} + O(1/ \log n).
\]
\end{theorem}
\begin{proof}
Let $T_n$ be the event that $\langle \pi_1, \dots, \pi_t\rangle$ is transitive. Then $T_n^c$ is the event that there is some subset $X \subset \Omega$ of size at most $n/2$ simultaneously fixed by $\pi_1, \dots, \pi_t$. Moreover if we take a minimal such $X$ then $\pi_1, \dots, \pi_t$ will act transitively on $X$. Thus $T_n^c$ coincides with the event $\bigcup_{k=1}^\floor{n/2} \{ N_k^* > 0\}$. By Lemma~\ref{N_k*-lemma} we have
\[
  P_\alpha\l( \bigcup_{k=2}^\floor{n/2} \{ N_k^* > 0 \} \r)
  \leq \sum_{k=2}^\floor{n/2} E_\alpha(N_k^*) = O(n^{-c}),
\]
provided $E_\alpha(N_1) \leq e^{-t-10} \log n$, while by Lemma~\ref{bonferroni} we have
\[
  P_\alpha(N_1^* > 0) = 1 - e^{-E_\alpha(N_1)} + O\l( 1/\log n \r)
\]
(noting $N_1 = N_1^*$). If $E_\alpha(N_1) \geq e^{-t-10} \log n$ then we can just bound
\[
  P_\alpha(T_n^c) \geq P_\alpha(N_1^* > 0) = 1 - O(n^{-c}) - O(1/\log n).
\]
We thus have in either case
\[
  P_\alpha(T_n^c) = 1 - e^{-E_\alpha(N_1)} + O\l( 1/\log n \r).
\]

Finally, we claim that at this level of approximation we can replace $E_\alpha(N_1)$ by simply $n (\alpha / n)^t$. By Lemma~\ref{N_k-lemma} we have
\[
  E_\alpha(N_1) = n \l( \frac{\alpha}{n+\alpha-1} \r)^t = n (\alpha / n)^t e^{O(t \alpha / n)}.
\]
If $E_\alpha(N_1) \geq \log n$ then both $e^{-E_\alpha(N_1)}$ and $e^{-n (\alpha / n)^t}$ are $O(1/n)$. If $E_\alpha(N_1) \leq \log n$ then $\alpha \leq O\l( n^{1-1/t} \log^t n\r)$, so the error in the approximation $e^{-E_\alpha(N_1)} \approx e^{-n (\alpha / n)^t}$ is $O(n^{-1/t + o(1)})$. This finishes the proof.
\end{proof}

Theorem~\ref{main-theorem} is immediate from Theorems~\ref{transitive} and \ref{intransitive}, at least if $\alpha \leq \epsilon n / \log^2 n$. If $\alpha \geq \epsilon n / \log^2 n$, Theorem~\ref{intransitive} implies that $\langle \pi_1, \dots, \pi_t \rangle$ is intransitive with probability $1 - O(1/\log n)$, and also $e^{-n(\alpha/n)^t} = O(1/\log n)$, so we don't need Theorem~\ref{transitive} in this case. Thus Theorem~\ref{main-theorem} is proved.

\section{The \texorpdfstring{$\alpha$}{alpha}-density of some other subgroups}

Essential to our calculation was the observation (Lemma~\ref{S_k-lemma}) that
\[
  P_\alpha(S_k \times S_{n-k}) = \frac{\alpha^{(k)} \alpha^{(n-k)}}{\alpha^{(n)}}.
\]
It is a little surprising that the $\alpha$-density of $S_k \times S_{n-k}$ has such a convenient formula. It turns out that there are ``$\alpha$-analogues'' of densities of at least a couple other standard subgroups of $S_n$ too.

Let $r$ be a divisor of $n$, and let
\[
  \Omega = \Omega_1 \cup \cdots \cup \Omega_r
\]
be a partition of $\Omega$ into $r$ sets $\Omega_i$ each of size $n/r$. The group of $\pi\in S_n$ preserving $\{\Omega_1, \dots, \Omega_r\}$ (possibly permuting them) is isomorphic to the wreath product
\[
  S_{n/r} \wr S_r.
\]
When the particular partition is understood or unimportant, we denote this subgroup simply $S_{n/r} \wr S_r$. This action of $S_{n/r} \wr S_r$ is called the \emph{imprimitive action} of the wreath product, and the groups $S_{n/r} \wr S_r$ are precisely the maximal imprimitive subgroups of $S_n$. Note that the index of $S_{n/r} \wr S_r$ in $S_n$ is
\[
  \frac{n!}{(n/r)!^r r!}.
\]

\begin{lemma}
We have
\[
  P_\alpha(S_{n/r} \wr S_r) = \frac{(n/r)!^r}{\alpha^{(n)}} \left( \frac{\alpha^{(n/r)}}{(n/r)!} \right)^{(r)}.
\]
\end{lemma}
\begin{proof}
There are $r!$ ways that a permutation $\pi \in S_n$ can preserve $\{\Omega_1, \dots, \Omega_r\}$: for each $\sigma \in S_r$ we might have
\[
  \pi(\Omega_i) = \Omega_{\sigma(i)} \qquad (i\in \{1, \dots, r\}).
\]
Fix $\sigma$. For each $i$ let $\pi_i = \pi|_{\Omega_i}$, so that
\[
  \pi_i: \Omega_i \to \Omega_{\sigma(i)}.
\]
Clearly choosing $\pi$ is the same as choosing $\pi_1, \dots, \pi_r$.

Let $i, \sigma(i), \dots, \sigma^{s-1}(i)$ be a cycle of $\sigma$, and let
\[
  \Psi = \Omega_i \cup \Omega_{\sigma(i)} \cup \dots \cup \Omega_{\sigma^{s-1}(i)}.
\]
Then the number of cycles of $\pi|_\Psi$ is the same as the number of cycles of the permutation of $\Omega_i$ defined by
\[
  \def\boxwidth{30pt}
  \Omega_i \xrightarrow{\mathmakebox[\boxwidth]{\pi_i}} \Omega_{\sigma(i)}
  \xrightarrow{\mathmakebox[\boxwidth]{\pi_{\sigma(i)}}} \cdots
  \xrightarrow{\mathmakebox[\boxwidth]{\pi_{\sigma^{s-2}(i)}}} \Omega_{\sigma^{s-1}(i)}
  \xrightarrow{\mathmakebox[\boxwidth]{\pi_{\sigma^{s-1}(i)}}} \Omega_i.
\]
In other words we have
\[
  C(\pi|_\Psi) = C(\pi_{\sigma^{s-1}(i)} \circ \cdots \circ \pi_i).
\]
Thus the $\alpha$-weighted count of $\pi|_\Psi$ inducing the given cycle is
\begin{align*}
  \sum_{\pi_i, \dots, \pi_{\sigma^{s-1}(i)}} \alpha^{C(\pi_{\sigma^{s-1}(i)} \circ \cdots \circ \pi_i)}
  &= \sum_{\pi_i, \dots, \pi_{\sigma^{s-2}(i)}} \sum_{\tau \in S_{n/r}} \alpha^{C(\tau)} \\
  &= (n/r)!^{s-1} \alpha^{(n/r)}.
\end{align*}

It follows that the $\alpha$-weighted count of all $\pi \in S_{n/r} \wr S_r$ inducing $\sigma$ is
\[
  (n/r)!^r \left( \frac{\alpha^{(n/r)}}{(n/r)!} \right)^{C(\sigma)}.
\]
Thus
\[
  P_\alpha \l( S_{n/r} \wr S_r \r)
  = \frac{(n/r)!^r}{\alpha^{(n)}} \sum_{\sigma \in S_r} \l( \frac{\alpha^{(n/r)}}{(n/r)!} \r)^{C(\sigma)}
  = \frac{(n/r)!^r}{\alpha^{(n)}} \l( \frac{\alpha^{(n/r)}}{(n/r)!} \r)^{(r)}. \qedhere
\]
\end{proof}

Here is another cute formula (cf.~\cite[Lemma~19]{brito_fowler_junge_levy_2018}):

\begin{lemma}
We have
\[
  P_\alpha(A_n) = \frac{1}{2} + \frac{1}{2} \frac{\alpha_{(n)}}{\alpha^{(n)}},
\]
where
\[
  \alpha_{(n)} = \alpha(\alpha-1) \cdots (\alpha - n + 1).
\]
\end{lemma}
\begin{proof}
We claim that
\[
  \sum_{\pi \in S_n} \op{sgn}(\pi) \alpha^{C(\pi)} = \alpha_{(n)}.
\]
Let $a_n$ denote the sum above, and let
\[
  f(X) = \sum_{n=0}^\infty \frac{a_n}{n!} X^n.
\]
Then
\begin{align*}
  f(X)
  &= \sum_{n=0}^\infty \sum_{\substack{c_1, \dots, c_n \geq 0 \\ \sum i c_i = n}} \frac{(-1)^{c_2 + c_4 + \cdots} \alpha^{c_1 + c_2 + \cdots}X^n}{\prod_{i=1}^n i^{c_i} c_i!} \\
  &= \prod_{i=1}^\infty \sum_{c=0}^\infty \frac{(-1)^{(i-1)c} \alpha^c X^{ic}}{i^c c!} \\
  &= \exp\l( \sum_{i=1}^\infty \frac{(-1)^{i-1} \alpha X^i}{i} \r) \\
  &= (1+X)^{\alpha} \\
  &= \sum_{n=0}^\infty \binom{\alpha}{n} X^n.
\end{align*}
This proves the claim.
\end{proof}

It's probably unreasonable to hope for many more such nice formulae for $\alpha$-densities of standard subgroups of $S_n$. The case of $S_{n^{1/k}} \wr S_k$ with its product action, for instance, appears to be much more complicated.

\bibliography{ewens}

\begin{thebibliography}{BFJL18}

\bibitem[ABT92]{arratia-barbour-tavare}
Richard Arratia, A.~D. Barbour, and Simon Tavar\'e.
\newblock Poisson process approximations for the {E}wens sampling formula.
\newblock {\em Ann. Appl. Probab.}, 2(3):519--535, 1992.

\bibitem[Bab89]{babai}
L\'aszl\'o Babai.
\newblock The probability of generating the symmetric group.
\newblock {\em J. Combin. Theory Ser. A}, 52(1):148--153, 1989.

\bibitem[BFJL18]{brito_fowler_junge_levy_2018}
Gerandy Brito, Christopher Fowler, Matthew Junge, and Avi Levy.
\newblock Ewens sampling and invariable generation.
\newblock {\em Combinatorics, Probability and Computing}, page 1–39, 2018.

\bibitem[Cra16]{crane}
Harry Crane.
\newblock The ubiquitous {E}wens sampling formula.
\newblock {\em Statist. Sci.}, 31(1):1--19, 2016.

\bibitem[Dix69]{dixon1}
John~D. Dixon.
\newblock The probability of generating the symmetric group.
\newblock {\em Math. Z.}, 110:199--205, 1969.

\bibitem[EFG17]{eberhard-ford-green-invariable}
Sean Eberhard, Kevin Ford, and Ben Green.
\newblock Invariable generation of the symmetric group.
\newblock {\em Duke Math. J.}, 166(8):1573--1590, 2017.

\bibitem[Ewe72]{ewens}
Warren~J. Ewens.
\newblock The sampling theory of selectively neutral alleles.
\newblock {\em Theoret. Population Biology}, 3:87--112; erratum, ibid. 3
  (1972), 240; erratum, ibid. 3 (1972), 376, 1972.

\bibitem[HT88]{hall-tenenbaum}
Richard~R. Hall and G\'erald Tenenbaum.
\newblock {\em Divisors}, volume~90 of {\em Cambridge Tracts in Mathematics}.
\newblock Cambridge University Press, Cambridge, 1988.

\bibitem[PPR16]{pemantle-peres-rivin}
Robin Pemantle, Yuval Peres, and Igor Rivin.
\newblock Four random permutations conjugated by an adversary generate sn with
  high probability.
\newblock {\em Random Structures \& Algorithms}, 49(3):409--428, 2016.

\end{thebibliography}
\bibliographystyle{alpha}
\end{document}